\documentclass{article}
\usepackage[all,ps,arc]{xy}
\usepackage{amsmath,amssymb,latexsym}
\usepackage{amsfonts}
\usepackage{amsthm}
\usepackage{xcolor}
\usepackage{cleveref}

\newtheorem{Theorem}{Theorem}[section]
\newtheorem{Lemma}[Theorem]{Lemma}
\newtheorem{Proposition}[Theorem]{Proposition}

\theoremstyle{definition}
\newtheorem{Definition}[Theorem]{Definition}
\theoremstyle{remark}
\newtheorem{Remark}[Theorem]{Remark}

\numberwithin{equation}{section}

\newcommand{\cA}{\ensuremath{\mathsf A}}
\newcommand{\cB}{\ensuremath{\mathsf B}}

\newcommand{\cG}{\ensuremath{\mathsf G}}
\newcommand{\cH}{\ensuremath{\mathsf H}}

\newcommand{\cQ}{\ensuremath{\mathsf Q}}

\newcommand{\cX}{\ensuremath{\mathsf X}}

\newcommand\Ker{\ensuremath{\text{Ker}}}
\newcommand\coker{\ensuremath{\mathrm{coker\,}}}
\newcommand\pr{\ensuremath{\text{pr}}}

\newcommand\CG{\ensuremath{2\mathbb{G}\mathsf{p}}}
\newcommand\SCG{\ensuremath{\mathbb{S}\mathsf{ym}2\mathbb{G}\mathsf{p}}}

\newcommand\Set{\ensuremath{\mathsf{Set}}}
\newcommand\Gp{\ensuremath{\mathsf{Gp}}}

\newcommand\XExt{\ensuremath{\mathsf{XExt}}}

\newcommand\BExt{\ensuremath{\mathsf{BExt}}}
\newcommand\OpExt{\ensuremath{\mathsf{OpExt}}}

\newcommand\Mod{\ensuremath{\mathsf{Mod}}}

\newcommand\Ab{\ensuremath{\mathsf{Ab}}}

\newdir{|>}{{}*!/8pt/@{|}*!/3.5pt/:(1,-.2)@^{>}*!/3.5pt/:(1,+.2)@_{>}}
\newdir{ >}{{}*!/-10pt/@{>}}
\newdir{ |>}{!/10pt/{}*!/4.5pt/@{|}*:(1,-.2)@^{>}*:(1,+.2)@_{>}}

\makeatletter
\@namedef{subjclassname}{Mathematics Subject Classification (2020)}
\makeatother

\begin{document}

\title{The third cohomology $2$-group}
\author{Alan S.~Cigoli$^1$ \and Sandra Mantovani$^2$\and Giuseppe Metere$^3$}

\date{\small
    $^1$ Dipartimento di Matematica ``Giuseppe Peano''\\ Universit\`a degli Studi di Torino \\[2ex]%
    $^2$ Dipartimento di Matematica ``Federigo Enriques''\\ Universit\`a degli Studi di Milano\\[2ex]%
   $^3$ Dipartimento di Matematica e Informatica\\  Universit\`a degli Studi di Palermo\\[2ex]
}







\maketitle

\begin{abstract}

In this paper we show that a finite product preserving  opfibration  can be factorized through an opfibration with the same property, but with groupoidal fibres. 
If moreover the codomain is additive, one can endow each fibre of the new opfibration  with a canonical symmetric $2$-group structure. We then apply such factorization to the opfibration that sends a crossed extension of a group $C$ to its corresponding $C$-module. The symmetric $2$-group structure so obtained on the fibres, defines the third cohomology $2$-group of $C$, with coefficients in a $C$-module. We show that the usual third and second cohomology groups are recovered as its homotopy invariants. Furthermore, even if all results are presented in the category of groups, their proofs are valid in any strongly protomodular semi-abelian category, once one adopts the corresponding internal notions.\\


\end{abstract}


\section{Introduction} \label{sec:intro}

For a fixed group $C$, let $\Mod(C)$ denote the category of $C$-modules, whose objects are pairs $(B,\xi)$, with $\xi\colon C\times B\to B$ a group action on the abelian group $B$, and morphisms are  equivariant  group homomorphisms.  $\Mod(C)$ is an abelian category, and it is well known to be equivalent to  the category $\Ab(\Gp/C)$ of abelian groups in $\Gp/C$  (also called  Beck $C$-modules, see \cite{Beck}). 

With any $C$-module $(B,\xi)$, we can associate the groupoid $\OpExt(C,B,\xi)$ of abelian extensions of $C$, with kernel $B$ and induced action $\xi$, where, given an extension of C with abelian kernel $B$
\[
\xymatrix{
0 \ar[r] & B \ar[r]^j & E \ar[r]^p & C \ar[r] & 0\,,
}
\]
the  action of $C$ on $B$ is given by the conjugation action of $E$, via any pointed set-theoretical section of  $p$. 

If we consider the set $\pi_0(\OpExt(C,B,\xi))$ of connected components,  we can extend the previous assignment to a functor  \[H^2(C,-)\colon  \Mod(C) \to \Set\] which \emph{preserves finite products}. Actually, likewise any $\Set$-valued functor with  abelian domain which preserves finite products, $H^2(C,-)$ factors through the category $\Ab$ of abelian groups (see \cite{BJ}, where the Baer sum of abelian extensions in a protomodular category is introduced in this way). As a consequence, the set $\pi_0(\OpExt(C,B,\xi))$ can be endowed with an  abelian group  structure, which is isomorphic to the usual  $H^2(C,B,\xi)$ (see \cite{Homology}).
 
It is well known since \cite{Gerstenhaber} and \cite{Holt} that also the cohomology functor $H^{3}(C,-)$ can be described in a similar fashion, i.e.\ using \emph{crossed extensions} (see \Cref{sec:XExt}):
\[
\xymatrix{
0 \ar[r] & B \ar[r]^-j & E_2 \ar[r]^\partial & E_1 \ar[r]^-p & C \ar[r] & 0\,.
}
\]
In this case the categories $\XExt(C,B,\xi)$ of such extensions are no longer groupoids. However, their connected components define again a $\Set$-valued  functor which preserves finite products. So, it is possible to apply the same argument as for $H^2(C,-)$.

In fact, the abelian group structure on $\pi_0(\OpExt(C,B,\xi))$ is nothing but the shadow of a richer structure of a symmetric 2-group $\cH^2(C,B,\xi)$ which can be defined directly on $\OpExt(C,B,\xi)$  (see \cite{CM16}).  This result has been anticipated by Bourn in \cite{Bourn99}, where he singles out some sufficient conditions to lift a monoidal closed structure to the fibres of a given opfibration. The groupoids $\OpExt(C,B,\xi)$ are actually  the fibres of an opfibration $P_C \colon \OpExt(C) \to \Mod(C)$ which preserves finite products. An articulate treatment of the subject can be found in \cite{groupal}, where  cartesian monoidal opfibrations \cite{MV2020} are investigated (see also \cite{Shulman}). It turns out that, when the codomain of a finite product preserving opfibration is additive, then its fibres are symmetric $2$-groups as soon as they are groupoids, and this is precisely what happens in the case of abelian extensions.   
 
 In the present paper we deal with the case of  opfibrations as before, but whose fibres are not groupoids. In \Cref{prop:fractions_cmo} we show that a $2$-group structure can still be defined on the fibres of a new opfibration obtained from the original one via a suitable category of fractions. 

The case of crossed extensions is exactly of the kind  described above. In fact, the categories $\XExt(C,B,\xi)$ turn out to be the fibres of an opfibration $\Pi_C\colon \XExt(C)\to\Mod(C)$ which again preserves finite products. The first step is to  investigate the category of fractions $[\BExt](C)$ of $\XExt(C)$ with respect to  $\Pi_C$-vertical maps (which are indeed weak equivalences between the underlying crossed modules, see \cite{AMMV13}).  Relying on results in \cite{categorical_ot} and \cite{Cat-frac}, in \Cref{thm:BExt} we show that such fractions can be described using diagrams called \emph{butterflies} by Noohi (see \cite{Noohi}).  By applying \Cref{prop:fractions_cmo}, we get that the functor $\Pi_C$ factors through a new cartesian monoidal opfibration $\overline\Pi_C$ with groupoidal fibres. Therefore, for each $C$-module $(B,\xi)$, the corresponding fibre of $\overline\Pi_C$ becomes a symmetric $2$-group which can be considered as the \emph{third cohomology $2$-group}, and will be denoted by $\cH^{3}(C,B,\xi)$.

 The abelian group  $\pi_0(\cH^{3}(C,B,\xi))$ of connected components of the $2$-group  $\cH^{3}(C,B,\xi)$ is isomorphic to the usual $H^{3}(C,B,\xi)$, as observed above. On the other hand, the abelian group  $\pi_1(\cH^{3}(C,B,\xi))$ given by the automorphisms of the monoidal identity  $I_\xi$ is isomorphic to the group $H^{2}(C,B,\xi)\cong \pi_0(\cH^{2}(C,B,\xi))$. This result is obtained as a consequence of \Cref{thm:I}, where it is proved that the $2$-group $\cH^{2}(C,B,\xi)$ is monoidally equivalent to the $2$-group $Eq(I_\xi)$ of auto-equivalences of $I_\xi$, with tensor product given by butterfly composition.

We remark here that, although all results are presented having as base category the category of groups, all proofs are indeed internal, in the sense that they can be performed in any strongly protomodular semi-abelian category (see \cite{BB}), once one adopts the corresponding internal notions.

\bigskip

We use \emph{sans-serif} capital letters for categories ($\cA, \cB$ etc.). 
Composition of arrows (1-cells) $f, g$ is denoted $g\cdot f$, or just $gf$. 
Binary products  in a category $\cB$ are denoted as usual, with projections $\pr_1,\pr_2$. 
We denote by $\CG$ the $2$-category of $2$-groups; its objects are (not necessarily strict) $2$-groups (also known as \emph{categorical groups}), morphisms are monoidal functors, and $2$-morphisms are monoidal natural transformations. Their symmetric monoidal version organizes in the $2$-category $\SCG$. There is a $2$-functor 
\[
(\pi_0,\pi_1)\colon \CG\to \Mod
\]
where the category of group modules $\Mod$ is considered as a $2$-discrete $2$-category. It associates with any $2$-group $\cG$ the group module $(\pi_0(\cG), \pi_1(\cG))$ of its \emph{homotopy invariants}, with $\pi_0(\cG)$  the group of connected components of $\cG$ and $\pi_1(\cG)$  the abelian group of the automorphisms of the identity object of $\cG$, endowed with a canonical $\pi_0(\cG)$-module structure (see for instance \cite{BL2004}). Notice that, if the $2$-group is symmetric, the group $\pi_0(\cG)$ is abelian. 
Finally, throughout the paper, we will use the additive notation for any group operation.

\section{Preliminaries} \label{sec:preliminaries}

In this section we recall from \cite{groupal} some basic results on monoidal opfibrations, and complete the description of the symmetric $2$-group  $\cH^2(C,B,\xi)$ given in the introduction.

\begin{Definition}\label{def:cart_mon_opfib}
	Let $\cX, \cB$ be categories with finite products, and $P\colon \cX\to\cB$ be an opfibration. We say that $P$ is \emph{cartesian monoidal} if it strictly preserves finite products and  the product of two cocartesian morphisms is still cocartesian.
\end{Definition}

In other words, cartesian monoidal opfibrations are nothing but monoidal opfibrations, with the monoidal structures given by  cartesian products, see  \cite[Definition 12.1]{Shulman}.

\begin{Proposition} \label{prop:cmo_groupoidal}
Let \cX\ and \cB\ be categories with finite products and $P\colon \cX \to \cB$ be an opfibration with groupoidal fibres. Then $P$ is a cartesian monoidal opfibration if and only if it strictly preserves finite products.
\end{Proposition}

\begin{proof}
It suffices to notice that each morphism in \cX\ factorizes through a cocartesian lifting of its image by $P$, with a $P$-vertical  comparison morphism, hence an isomorphism by assumption. So any morphism, in particular any product of morphisms, is cocartesian.
\end{proof}

If the codomain $\cB$ is an additive category, every object of $\cB$ is endowed with a (unique) commutative monoid structure. In this case, the  monoidal opfibration induces on its fibres and on change-of-base functors  canonical symmetric monoidal structures, as shown in \cite[Theorem 4.2]{MV2020} (see also \cite[Remark 4.8]{groupal}).  Furthermore, we have the following result, see \cite[Theorem 5.9]{groupal}.

\begin{Theorem}\label{thm:fibres_groupoids}
	Let $P\colon \cX\to \cB$ be a cartesian monoidal opfibration. If the category $\cB$ is additive, then the symmetric monoidal structures induced on the  fibres are $2$-groups  if and only if the fibres are groupoids.	
\end{Theorem}

Let $C$ be a group, and $\OpExt(C)$ be the category of abelian extensions of the group $C$ and their morphisms. The functor $P\colon \OpExt(C)\to \Mod(C)$, which assigns to each abelian extension the canonical $C$-module structure on the kernel, is an instance of Bourn's direction functor (see \cite{Bourn99}). In \emph{op.cit.}, the author proved that the latter is an opfibration that preserves finite products and that cocartesian maps are stable under binary products. In other words, it is cartesian monoidal. As explained in \cite[Section 6.1]{groupal}, since the codomain of $P$ is additive, thanks to Theorem \ref{thm:fibres_groupoids}, each $P$-fibre $\OpExt(C,B,\xi)$ inherits a symmetric $2$-group structure:
\[
\cH^2(C,B,\xi):=(\OpExt(C,B,\xi),\otimes,E_{\rtimes})
\]
where the tensor product $\otimes$ is obtained by the classical construction used to define Baer sums, \emph{before taking the isomorphism classes of extensions}. The identity $E_\rtimes$ is the canonical extension associated with the semidirect product  $B\rtimes_{\xi}C$.  Moreover, each $C$-module morphism $f\colon B\to B'$ induces a symmetric monoidal functor
\[
f_*\colon \cH^2(C,B,\xi)\to \cH^2(C,B',\xi')\,.
\]
Concerning the homotopy invariants $\pi_0,\pi_1$, as recalled in the introduction, we have that \[\pi_0(\cH^2(C,B,\xi))\cong H^2(C,B,\xi)\,.\] On the other hand, the abelian group  $\pi_1(\cH^2(C,B,\xi))$ consists of the automorphisms of the  extension $E_{\rtimes}$. According to \cite[Proposition IV.2.1]{Homology}, this can be interpreted as the group of 1-cocycles, i.e. \[\pi_1(\cH^2(C,B,\xi))\cong Z^1(C,B,\xi)\,.\]

\section{Crossed extensions} \label{sec:XExt}

Recall that a \emph{crossed module} consists of a group homomorphism $\partial \colon E_2\to E_1$, endowed with an action of $E_1$ on $E_2$, satisfying the following conditions:
\begin{align*}
\text{i.} & \quad\partial(g\ast x)=g+\partial x-g & \text{(Pre-crossed module)} \\
\text{ii.} & \quad \partial x_1\ast x_2=x_1+x_2-x_1 & \text{(Peiffer)}
\end{align*}

Crossed modules, and their usual morphisms, organize in a category, which is equivalent to the category of internal groupoids in the category of groups. To fix notations, let us recall that each crossed module $\partial$ corresponds to the internal groupoid
\begin{equation} \label{diag:xmod_gpd}
\xymatrix{
E_2 \rtimes E_1 \ar@<1ex>[r]^-d \ar@<-1ex>[r]_-c & E_1, \ar[l]
}
\end{equation}
where the semidirect product operation is induced by the given action of $E_1$ on $E_2$ and the morphisms $d$ and $c$ are as follows:
\begin{align*}
d(x,g) & = \partial x+g, \\
c(x,g) & = g.
\end{align*}

Let $C$ be a group. A crossed extension $E$ of $C$ is given by
\[
\xymatrix{
 0 \ar[r] & B \ar[r]^{j} & E_2  \ar[r]^{\partial} & E_1 \ar[r]^{p}  & C  \ar[r] & 0
}
\]
where $\partial$ is endowed with a crossed module structure, with $B$ and $C$  fixed kernel and cokernel, respectively. Morphisms of crossed extensions of $C$ are given by commutative diagrams:

\[
\xymatrix{
 0 \ar[r] & B \ar[r]^{j} {\ar[d]_{\beta}} & E_2
{\ar[d]_{f_2}} \ar[r]^{\partial} & E_1 \ar[r]^{p} {\ar[d]_{f_1}} & C {\ar@{=}[d]} \ar[r] & 0 \\
 {0} {\ar[r]} &{B'} {\ar[r]^{j'}} & {E_2'} {\ar[r]^{\partial'}} &{E_1'} {\ar[r]^{p'}} & {C} {\ar[r]} & {0}
}
\]
where  the middle square is a crossed module morphism. We denote by $\XExt(C)$ the category  of \emph{crossed} extensions of $C$. 

In each crossed extension the kernel $B$ is central and it comes equipped with a $C$-module structure $(B,\xi)$, where $\xi$ is the action of $C$ over $B$ induced by  the action of $E_1$ on $E_2$ (see for instance \cite{KBrown}). 

One can lift a given crossed extension  along a $C$-module morphism $\beta\colon(B,\xi) \to (B',\xi')$ by means of what is called the \emph{push forward} along $\beta$ (see \cite{butterflies}, or \cite{pf} for the semi-abelian version):
\[
\xymatrix{
0 \ar[r] & B \ar[r]^j \ar[d]_{\beta} {\ar@{}[dr]|{\text{p.f.}}} & E_2 \ar[r]^{\partial} {\ar[d]} & E_1 \ar[r]^p {\ar@{=}[d]} & C \ar[r] {\ar@{=}[d]} & 0 \\
0 {\ar[r]} & B' {\ar[r]} & {B'\times^B E_2} {\ar[r]^-{\partial'}} & {E_1}{\ar[r]^p} & {C} {\ar[r]} &{0}
}
\]
Such liftings are universal, i.e.\  the functor
$
P_C\colon \XExt(C)\to \Mod(C)
$
is an opfibration (see \cite[Section 4]{yoneda}). Actually,  it is also cartesian monoidal. A proof in the context of strongly semi-abelian categories, based on properties of Bourn's 1-dimensional direction functor,  can be found in \cite{Rodelo2009}.

Indeed, $P_C$ has not groupoidal fibres, since  morphisms in the fibre over $(B,\xi)$ are of the kind $(1,f_1,f_2,1)$:
\[
\xymatrix{
0 \ar[r] & B\ar[r]^{j} \ar@{=}[d] & E_2 \ar[r]^{\partial} \ar[d]_{f_2} & E_1 \ar[r]^{p} \ar[d]^{f_1} & C \ar@{=}[d] \ar[r] & 0 \\
0 \ar[r] & B \ar[r]^{j'} & E_2' \ar[r]^{\partial'} & E_1' \ar[r]^{p'} & C \ar[r] & 0
}
\]
which induce weak equivalences of crossed modules (see \cite{AMMV13}). Such maps do not have inverses in general, so that it is not possible to endow the fibres of $P_C$ with a 2-group structure. The idea is to turn $P_C$ into an opfibration with groupoidal fibres, but still cartesian monoidal. This will be performed the next section.

\section{The symmetric $2$-group $\cH^3(C,B,\xi)$}

The categorical construction that we need in order to make  $P_C$-vertical morphisms invertible consists in taking the corresponding category of fractions. We rely on results obtained in \cite{AMMV13}, where such fractions are described by means of suitable diagrams called \emph{butterflies} (see \cite{Noohi}).

\bigskip
A butterfly between two crossed extensions of $C$ is depicted as a diagram
\begin{equation}\label{diag:butterfly}
\begin{aligned}
\xymatrix{
0 \ar[r] & B \ar[r]^{j} & E_2 \ar[rr]^{\partial} \ar[dr]_{\kappa} & & E_1 \ar[r]^{p} & C \ar[r] & 0 \\
& & & F \ar[dr]_{\gamma} \ar[ur]_{\delta} \\
0 \ar[r] & B' \ar[r]_{j'} & E_2' \ar[rr]_{\partial'} \ar[ur]_{\iota} & & E_1'\ar[r]_{p'} & C \ar[r] & 0
}	
\end{aligned}
\end{equation}
where the following conditions are satisfied:
\begin{itemize}
\item[i.] $\delta\cdot\kappa=\partial$, \,  $\gamma\cdot\iota=\partial'$ and $p\cdot \delta = p'\cdot\gamma$,
\item[ii.] $(\kappa,\gamma)$ is a complex, i.e.\ $\gamma\cdot \kappa = 0$ and $(\delta,\iota)$ is a short exact sequence, i.e.\ $\delta=\coker \iota$ and $\iota=\ker\delta$,
\item[iii.] the action of $F$ on $E_2$, induced by the one of $E_1$ on $E_2$ via $\delta$, makes $\kappa$ a pre-crossed module,
\item[iv.] the action of $F$ on $E_2'$, induced by the one of $E_1'$ on $E_2'$ via $\gamma$, makes $\iota$ a pre-crossed module.
\end{itemize}
We will use the short notation $\widehat{F}\colon E \to E'$ to denote the above butterfly. 

With each butterfly $\widehat{F}\colon E\to E'$ as in \eqref{diag:butterfly}, one can associate a $C$-module morphism $\beta\colon B \to B'$. In order to describe how it is constructed, let us recall from \cite{AMMV13} that each butterfly induces a span of crossed module morphisms as in the right-hand side of the following diagram:
\begin{equation} \label{diag:bfly_span}
\begin{aligned}
\xymatrix{
& B \ar@{=}[dl] \ar[r]^j & E_2 \ar[rr]^{\partial} \ar@{-->}[dr]_{\kappa} & & E_1 \\
B \ar[dr]_\beta \ar[r]^(.5){\overline{j}} & E_2 \times E_2' \ar[ur]^{\pr_1} \ar[dr]_(.4){\pr_2} \ar[rr]^{\kappa\sharp\iota} & & F \ar[dr]_{\gamma} \ar[ur]_{\delta} \\
& B' \ar[r]^(.4){j'} & E_2' \ar[rr]^{\partial'} \ar@{-->}[ur]_{\iota} & & E_1'
}
\end{aligned}
\end{equation}
where the crossed module $\kappa\sharp\iota$ is the cooperator (see \cite{BB}) of the arrows $\iota$ and $\kappa$, which exists since the images  $\kappa(E_2)$ and $\iota(E_2')$ commute in $F$ (see \cite{AMMV13} for details). More explicitly, $\kappa\sharp\iota(x,x')=\kappa x +\iota x'$. Now, the square $\delta\cdot\kappa\sharp\iota = \partial\cdot\pr_1$ turns out to be a pullback, so that one can choose $B$ as a kernel of $\kappa\sharp\iota$ via a morphism $\overline{j}$ with  $\pr_1\cdot\overline{j}=j$. We call $\beta\colon B \to B'$ the unique morphism such that $j'\cdot\beta=\pr_2\cdot\overline{j}$, whence $\overline{j}(b)=(j(b),j'\beta(b))$.

Actually, $\beta$ is a $C$-module morphism, so it admits an opposite $-\beta$. Notice that
\begin{equation} \label{eq:-beta}
\kappa\cdot j = \iota\cdot j'\cdot (-\beta).
\end{equation}
One can prove this by showing that $\kappa\cdot j + \iota\cdot j'\cdot \beta = 0$. The same argument holds in any semi-abelian category (see \cite[\S 1]{BB} for  details on symmetrizable morphisms). Consider the commutative diagram
\[
\xymatrix{
& & & & E_2 \\
B \ar@/^2ex/[urrrr]^j \ar@/_2ex/[drrrr]_{j'\beta} \ar[rr]|-{\langle 1,1 \rangle} & & B \times B \ar[r]^{1\times \beta} & B\times B' \ar[r]^{j\times j'} & E_2 \times E_2' \ar[u]_{\pr_1} \ar[d]^{\pr_2} \ar[r]^-{\kappa\sharp\iota} & F \\
& & & & E_2'
}
\]
The composite $\kappa\sharp\iota\cdot(j\times j')\cdot(1\times\beta)$ is the cooperator of $\kappa\cdot j$ with $\iota\cdot j'\cdot\beta$, so that precomposition with the diagonal $\langle 1,1\rangle\colon B \to B\times B$ yields the sum $\kappa\cdot j + \iota\cdot j'\cdot \beta$ (see Definition 1.3.21 in \cite{BB}). This sum is trivial since $(j\times j')\cdot (1\times\beta)\cdot\langle 1,1\rangle = \overline{j}$, the two terms being equal by composition with product projections.

After our discussion above, we are allowed to report more information in the picture of a butterfly. For instance, diagram \eqref{diag:butterfly} becomes:
\[
\xymatrix{
0 \ar[r] & B\ar@{-->}[dd]_{\beta} \ar[r]^{j} & E_2 \ar[rr]^{\partial} \ar[dr]_{\kappa} & & E_1 \ar[r]^{p} & C\ar@{=}[dd] \ar[r] & 0 \\
& & & F \ar[dr]_{\gamma} \ar[ur]_{\delta} \\
0 \ar[r] & B' \ar[r]_{j'} & E_2' \ar[rr]_{\partial'} \ar[ur]_{\iota} & & E_1'\ar[r]_{p'} & C \ar[r] & 0
}	\]
where $\beta$ is dashed in order to remind the reader that the  pentagon on the left only commutes up to a $-1$ factor. Notice that, on the other hand, the pentagon on the right does commute, since $p\cdot\delta=p'\cdot\gamma$ by hypotesis.

\smallskip
Butterflies between crossed extensions can be composed, and details can be found in \cite{AMMV13}. Here we just describe the construction of the composite butterfly. Let us consider two butterflies $\widehat{F}\colon E \to E'$ and $\widehat{F}'\colon E' \to E''$. The following diagram illustrates the composition  $\widehat{F}'\cdot \widehat{F}$ (the kernels and cokernels of the involved crossed modules are omitted in the diagram, for the sake of clarity)
\begin{equation}\label{diag:butterfly_composition}
\begin{aligned}
\xymatrix{
&& E_2\ar[ddl]_{\langle\kappa,0\rangle} \ar@{-->}[rr]_{\partial} \ar[dr]_{\kappa} & & E_1  \\
&& & F \ar[dr]_{\gamma} \ar[ur]_{\delta} \\
F'\cdot F\ar@/^10ex/[rrrruu]^{\overline{\delta\cdot p_1}}  \ar@/_10ex/[rrrrdd]_{\overline{\gamma'\cdot p_2}}
&\ar[l]_{q}F\times_{E_1'}F' \ar[drr]_{p_2}\ar[urr]^{p_1}  & \ar[l]_(.3){\langle\iota, \kappa' \rangle} E_2' \ar@{-->}[rr]_{\partial'}\ar[dr]^{\kappa'} \ar[ur]_{\iota} & & E_1'\\
&& & F' \ar[dr]_{\gamma'} \ar[ur]_{\delta'} \\
&& E_2'' \ar[uul]^{\langle0,\iota'\rangle} \ar@{-->}[rr]^{\partial''} \ar[ur]_{\iota'} & & E_1'' 
}	
\end{aligned}
\end{equation}
In diagram above, $F\times_{E_1'}F'$ is the pullback over the pair $(\gamma,\delta')$, $q$ is the cokernel of $\langle\iota, \kappa' \rangle$,   $\overline{\gamma'\cdot p_2}$ and $\overline{\delta\cdot p_1}$ are the unique arrows such that $\overline{\gamma'\cdot p_2}\cdot q =\gamma'\cdot p_2$ and $\overline{\delta\cdot p_1}\cdot q= \delta\cdot p_1$ respectively. The resulting butterfly $\widehat{F'\cdot F}\colon E\to E''$ is given by  the complex  $(q\cdot \langle\kappa,0\rangle,\overline{\gamma'\cdot p_2})$ and the short exact sequence $(q\cdot \langle0,\iota'\rangle, \overline{\delta\cdot p_1})$.

Identity butterflies are depicted as follows:
\begin{equation}\label{diag:id_butterfly}
\begin{aligned}
\xymatrix{
0 \ar[r] & B\ar@{==}[dd] \ar[r]^{j} & E_2 \ar[rr]^{\partial} \ar[dr]_{\ker (d)} & & E_1 \ar[r]^{p} & C \ar[r]\ar@{=}[dd]  & 0 \\
& & & E_2\rtimes E_1 \ar[dr]_{d} \ar[ur]_{c} \\
0 \ar[r] & B \ar[r]_{j} & E_2 \ar[rr]_{\partial} \ar[ur]_{\ker(c)} & & E_1\ar[r]_{p} & C \ar[r] & 0
}	
\end{aligned}
\end{equation}
where $c$ and $d$ are as in (\ref{diag:xmod_gpd}). 

\begin{Definition}
We say that two butterflies $\widehat{F},\widehat{F'}\colon E \to E'$ are isomorphic to each other if there exists an isomorphism $\sigma\colon F \to F'$ such that
\[
\sigma\cdot \iota = \iota',\quad \sigma\cdot\kappa = \kappa', \quad \gamma'\cdot\sigma = \gamma,\quad \delta'\cdot\sigma = \delta.
\]
\end{Definition}
By taking isomorphism classes of butterflies as morphisms between crossed extensions of $C$, we get a category $[\BExt](C)$. From Proposition 6.4 in \cite{categorical_ot} and Proposition 3.14 in \cite{Cat-frac} we obtain the following result.  

\begin{Theorem} \label{thm:BExt}
The category $[\BExt](C)$ is the \emph{category of fractions} of $\XExt(C)$ with respect to $P_C$-vertical maps, by means of a functor $ Q_C\colon \XExt(C) \to [\BExt](C) $ which is the identity on objects.
\end{Theorem}

\begin{Remark}
A description of $Q_C$ on morphisms can be found in \cite[Section 6.4]{categorical_ot}. Those butterflies representing a class in the image of $Q_C$ are called \emph{representable}.
On the other hand, as proved in \cite{CM16}, those butterflies whose isomorphism class is invertible in $[\BExt](C)$ are precisely the ones where also $(\gamma,\kappa)$ is a short exact sequence. These are called \emph{flippable} butterflies.

Actually, $[\BExt](C)$ is the classifying category of the bicategory of butterflies of crossed extensions of $C$. It was proved in \cite{AMMV13} that butterflies provide a bicategory of fractions of crossed modules with respect to weak equivalences, so that the above theorem can be seen as a shadow of this result. 
\end{Remark}

Theorem \ref{thm:BExt} is important in our context also because it shows that the \emph{category of fractions} of $\XExt(C)$  we need, is still a locally small category. Going back to the functor $P_C\colon\XExt(C)\to\Mod(C)$, recall that it sends $P_C$-vertical maps to isomorphisms. Hence, by the universal property of $Q_C$, we get a factorization of $P_C$ given by
\begin{equation} \label{diag:Pc_bar}
\begin{aligned}
\xymatrix{
\XExt(C) \ar[r]_{Q_C} \ar@/^4ex/[rr]^{P_C} & [\BExt](C) \ar[r]_{\overline{P}_C} & \Mod(C),
}
\end{aligned}
\end{equation}
where $\overline{P}_C$ associates with each isomorphism class of butterflies the $C$-module morphism $\beta$ described in diagram \eqref{diag:bfly_span} which is invariant under isomorphisms (see Theorem 6.6 of \cite{categorical_ot} for details).

\begin{Proposition} \label{prop:fact_gpd}
In the factorization of diagram \eqref{diag:Pc_bar}, $\overline{P}_C$ is an opfibration with groupoidal fibres.
\end{Proposition}

\begin{proof}
The thesis follows by Theorem 4.3 of \cite{Cat-frac}, since, in 2-categorical terms, $Q_C$ is  the coinverter of the identee of $P_C$.
\end{proof}

Now we are going to show that $\overline{P}_C$ inherits from $P_C$ the property of being cartesian monoidal. As for the preservation of finite products, we need a preliminary Lemma, which follows from results in Section 3.2 of \cite{KLW} (see also \cite{Day}).

\begin{Lemma} \label{lemma:fractions_pres_products}
Let $\cA$ be a category with finite products and $\Sigma$ a class of morphisms in $\cA$ which contains identities and such that, if $f$ and $g$ are in $\Sigma$, $f\times g$  is in $\Sigma $ as well. Then the category of fractions $\cA[\Sigma^{-1}]$ has finite products and the localization functor $Q\colon \cA \to \cA[\Sigma^{-1}]$ preserves them.
\end{Lemma}

As an application, we get the following general result.

\begin{Proposition} \label{prop:fractions_cmo}
Let \cX\ and \cB\ be categories  with finite products, and $P\colon \cX \to \cB$ an opfibration strictly preserving them. Consider the factorization
\[
\xymatrix@!=5ex{
\cX \ar@/^3ex/[rr]^{P} \ar[r]_-{Q} & \cQ \ar[r]_-{\overline{P}} & \cB
}
\]
of $P$ through  the category of fractions $(Q,\cQ)$, with respect to the class of $P$-vertical morphisms. Then $\overline{P}\colon \cQ \to \cB$ is a cartesian monoidal opfibration with groupoidal fibres.

If moreover $\cB$ is additive, each fibre of $\overline P$ is endowed with a symmetric $2$-group structure, and  change-of-base functors are symmetric monoidal.
\end{Proposition}

\begin{proof}
Lemma \ref{lemma:fractions_pres_products} is applicable, since $P$-vertical morphisms are closed under finite products. Hence, finite products in \cX\ serve also as finite products in \cQ. So $\overline{P}$ preserves finite products as soon as $P$ does. Then by Theorem 4.3 in \cite{Cat-frac} and \Cref{prop:cmo_groupoidal}, $\overline{P}$ is a cartesian monoidal opfibration with groupoidal fibres. Finally, when $\cB$ is additive, the  result follows from \Cref{thm:fibres_groupoids}.
\end{proof}

As observed in \Cref{sec:XExt}, $P_C\colon\XExt(C)\to\Mod(C)$ is a cartesian monoidal opfibration, hence we can apply \Cref{prop:fractions_cmo} to get that ${\overline{P}_C}: [\BExt](C) \to \Mod(C)$ is a cartesian monoidal opfibration with groupoidal fibres. Moreover, the codomain $\Mod(C)$ is an additive category, so by \Cref{thm:fibres_groupoids} each fibre of $\overline{P}_C$ becomes a symmetric 2-group, which can be considered as \emph{the third cohomology 2-group} of $C$ with coefficients in $(B,\xi)$ and denoted by ${\cH^3(C,B,\xi)}$.

It is easy to see that if one considers connected components, the tensor operation induces precisely the Baer sum, which makes $\pi_0(\cH^3(C,B,\xi))$ a group (isomorphic to $H^3(C,B,\xi)$).

As a matter of fact, for a given crossed extension
\[
\xymatrix{
E\colon & 0 \ar[r]
& B \ar[r]^{j}
& E_2 \ar[r]^{\partial}
& E_1 \ar[r]^{p}
& C \ar[r]
& 0,
}
\]
any representative of the inverse of the class $[E]$ in $\pi_0(\cH^3(C,B,\xi))$ is a pseudo-inverse of $E$ in $\cH^3(C,B,\xi)$. For example, one can take
\[
\xymatrix{
E^*\colon & 0 \ar[r]
& B \ar[r]^{-j}
& E_2 \ar[r]^{\partial}
& E_1 \ar[r]^{p}
& C \ar[r]
& 0\,.
}
\]
To the reader's convenience, we provide here a proof of this fact. First, we have to compute the tensor product $E\otimes E^*$  of the two crossed extensions above, then find a flippable butterfly in the fibre over $(B,\xi)$ between $E\otimes E^*$ and the unit object $I_\xi $:
\[
\xymatrix{
0 \ar[r] & B \ar@{=}[r] & B \ar[r]^0 & C \ar@{=}[r] & C \ar[r] & 0
}
\]
of ${\cH^3(C,B,\xi)}$. This process is summarized in the diagram below:

\smallskip
\[
\xymatrix@C=5ex{
0 \ar[r] & B\times B \ar[d]_{[1,1]} \ar[r]^-{j\times(-j)} \ar@{}[dr]|{\text{p.f.}} & E_2 \times E_2 \ar[d]^{f_2} \ar[rr]^-{\partial \times \partial} \ar@/^/@{-->}[ddr]^(.4){\varphi} & & E_1 \times_C E_1 \ar[r]^-{p\cdot \pr_2} \ar@{=}[d] & C \ar[r] \ar@{=}[d] & 0 \\
0 \ar[r] & B \ar[r]^-{j'} \ar@{==}[dd] \ar@/_/@{-->}[drr]_(.4){\ker(c)\cdot(-j)} & \Ker(p\cdot c) \ar[rr]^{\partial'} \ar[dr]|{\ker(p\cdot c)} & & E_1 \times_C E_1 \ar[r]^-{p\cdot \pr_2} & C \ar@{=}[dd] \ar[r] & 0 \\
& & & E_2 \rtimes E_1 \ar[dr]^{p\cdot c} \ar[ur]_{\langle c,d \rangle} \\
0 \ar[r] & B \ar@{=}[r] & B \ar[rr]^{0} \ar[ur]^{\ker(c)\cdot j} & & C \ar@{=}[r] & C \ar[r] & 0
}
\]

\smallskip
\noindent The arrow $\varphi\colon E_2 \times E_2 \to E_2 \rtimes E_1$ is the cooperator of $\ker(d)\colon E_2\to E_2 \rtimes E_1$ with $\ker(c)\colon E_2\to E_2 \rtimes E_1$, where $d$ and $c$ are as in diagram \eqref{diag:xmod_gpd}. More explicitly, $\varphi(x_1,x_2)=(-x_1+{\partial x_1}\ast x_2,\partial x_1)$. It is easy to check that $\varphi\cdot(j\times(-j))=\ker(c)\cdot(-j)\cdot[1,1]$. Moreover, both the dashed arrows factor uniquely through the kernel of $p\cdot c$, hence we get the upper left commutative diagram.
Now consider the commutative diagram
\[
\xymatrix{
B \ar@{=}[d] \ar[rr]^-{j'} & & \Ker(p\cdot c) \ar@{}[dr]|{(a)} \ar[d]|{\ker(p\cdot c)} \ar[r]^{\overline{\partial}} & \Ker(p\cdot \pr_2) \ar@{}[dr]|{(b)} \ar[d]|{\ker(p\cdot \pr_2)} \ar[r] & 0 \ar[d] \\
B \ar[rr]^-{\ker(c)\cdot (-j)} & & E_2 \rtimes E_1 \ar[r]^-{\langle c,d \rangle} & E_1 \times_C E_1 \ar[r]^-{p\cdot \pr_2} & C.
}
\]
Both squares $(b)$ and $(a)+(b)$ are pullback, so by cancellation $(a)$ is a pullback. Hence, since $\ker(c)\cdot (-j)$ is a kernel of $\langle c,d \rangle$, $j'$ is a kernel of $\overline{\partial}$. Notice that $\langle c,d \rangle$ is a regular epimorphism, since it is the comparison arrow of the internal groupoid $E_2 \rtimes E_1$ to its support $E_1 \times_C E_1$ (i.e.\ the kernel pair of the coequalizer $p$ of $d$ and $c$). Let us denote $\partial'=\langle c,d \rangle\cdot \ker(p\cdot c)$. One can check that $\ker(p\cdot \pr_2)\cdot\overline{\partial}\cdot f_2 = \langle c,d \rangle\cdot\varphi = \partial\times\partial$. But $\partial\times\partial = \ker(p\cdot\pr_2)\cdot\overline{\partial\times\partial}$, where $\overline{\partial\times\partial}$ is a cokernel of $j\times(-j)$, so that $\overline{\partial}\cdot f_2 = \overline{\partial\times\partial}$.

The observations above explain that the commutative square  $f_2\cdot (j\times(-j)) = j'\cdot[1,1]$ induces isomorphisms on kernels and cokernels of the horizontal arrows, hence it is a push forward, thanks to \cite[Theorem 2.13]{pf}. As a consequence, by construction $(p\cdot \pr_2,\partial',j')$ is the tensor product $E\otimes E^*$. Furthermore, we obtain an isomorphism, represented by the flippable butterfly $\widehat{E_2 \rtimes E_1}$, between $E\otimes E^*$ and the unit object $I_\xi$.
\medskip

As recalled in the introduction, associated with any 2-group $\cG$, beside $\pi_0(\cG)$ there is also the abelian group $\pi_1(\cG)$ of automorphisms of its unit object. Let us investigate it in our context.

It was proved in \cite{CM16} that elements in $\pi_1(\cH^3(C,B,\xi))$ are exactly (isomorphism classes) of  butterflies of the form:
\begin{equation}\label{diag:butt_ext}	
\begin{aligned}
\xymatrix{
0\ar[r]
& B\ar@{=}[r]\ar@{==}[dd]
& B\ar[rr]^{0}\ar[dr]_{-\kappa}
&
& C\ar@{=}[r]
& C\ar@{=}[dd]\ar[r]
& 0
\\
& & & E \ar[dr]^{\gamma} \ar[ur]_{\gamma}& & &\\
0\ar[r]
& B\ar@{=}[r]
& B\ar[rr]_{0}\ar[ur]^{\kappa}
&
& C\ar@{=}[r]
& C\ar[r]
& 0
}
\end{aligned}
\end{equation}
As observed in \cite{CM16}, both $\pi_1(\cH^3(C,B,\xi))$ and $\pi_0(\cH^2(C,B,\xi))$ are isomorphic to the abelian group $H^2(C,B,\xi)$. More is true:  by defining $\Phi(\kappa,\gamma)$ as the butterfly in \eqref{diag:butt_ext}, one gets a monoidal equivalence 
\[\Phi\colon \cH^2(C,B,\xi)\to  \mathsf{Eq}(I_\xi)\,,\] 
 where $\mathsf{Eq}(I_\xi)$ is the $2$-group with objects the butterflies as in \eqref{diag:butt_ext}, and arrows the isomorphisms of such butterflies, with tensor product given by butterfly composition. Indeed, since $\pi_0(\mathsf{Eq}(I_\xi))$ is isomorphic to $\pi_1(\cH^3(C,B,\xi))$, $\pi_0(\Phi)$ makes the latter isomorphic to $\pi_0(\cH^2(C,B,\xi))$, as recalled above.

\begin{Theorem}\label{thm:I}
	Given a $C$-module $B$ with action $\xi$, the assignment $\Phi$ gives rise to a symmetric strict monoidal equivalence between the symmetric $2$-groups $\cH^2(C,B,\xi)$ and $\mathsf{Eq}(I_\xi)$. 
\end{Theorem} 
\begin{proof}
Let us recall the construction of the Baer sum in $\cH^2(C,B,\xi)$ of two abelian extensions 
\[
E\colon \xymatrix{
B\ar[r]^{\kappa}&E\ar[r]^{\gamma}&C
}
\qquad 
E'\colon \xymatrix{
B\ar[r]^{\kappa'}&E'\ar[r]^{\gamma'}&C
}
\]
of cokernel $C$ and abelian kernel $B$. In order to get their Baer sum $E\oplus E'$ (displayed in the left-most vertical sequence of maps, diagram (\ref{diag:Baer_of_ext}) below)

\begin{equation}\label{diag:Baer_of_ext}	
\begin{aligned}
\xymatrix@C=12ex{
B\ar[d]_{\bar \kappa}\ar@{}[dr]|{(a)}
&B\times B\ar[l]_{[1,1]}\ar@{=}[r]\ar[d]^{\kappa\times_C \kappa'}
&B\times B\ar[d]^{\kappa\times \kappa'}
\\
E\oplus E'\ar[d]_{\bar\gamma}
& E\times_{C}E'\ar[l]_p \ar[r]^{\langle p_1,p_2\rangle}\ar[d]_r\ar@{}[dr]|{(b)}
& E\times E'\ar[d]^{\gamma\times \gamma'}
\\
C\ar@{=}[r]
&C\ar[r]_-{\langle1,1\rangle}
& C\times C
}	
\end{aligned}
\end{equation}
first take the pullback $(b)$ of $\gamma \times \gamma'$ along the diagonal map $\langle1,1\rangle$ of $C$ (i.e., take the pullback of $\gamma$ along $\gamma'$), then take the pushforward $(a)$ of the kernel $\kappa\times_C\kappa'$ along the codiagonal $[1,1]$ of the abelian object $B$. The short exact sequence $(\bar\kappa,\bar\gamma)$ is the Baer sum of $(\kappa,\gamma)$ and $(\kappa',\gamma')$.

On the other hand, if we consider $\Phi(E)$ and $\Phi(E')$ (as in  diagram (\ref{diag:butt_ext})), and we take their composite as butterflies, we get the following construction: 
\begin{equation}\label{diag:Butter_of_ext}	
\begin{aligned}
\xymatrix{
&& B\ar[ddl]_{\langle-\kappa,0\rangle} \ar@{->}[rr]_{0} \ar[dr]_{-\kappa} & & C  \\
&& & E \ar[dr]_{\gamma} \ar[ur]_{\gamma} \\
E'\cdot E\ar@/^10ex/[rrrruu]^{\overline{\gamma\cdot p_1}}  \ar@/_10ex/[rrrrdd]_{\overline{\gamma'\cdot p_2}}
&\ar[l]_{q}E\times_{C}E' \ar[drr]_{p_2}\ar[urr]^{p_1}  & \ar[l]_(.3){\langle\kappa, -\kappa' \rangle} B \ar@{->}[rr]_{0}\ar[dr]^{-\kappa'} \ar[ur]_{\kappa} & & C\\
&& & E' \ar[dr]_{\gamma'} \ar[ur]_{\gamma'} \\
&& B \ar[uul]^{\langle0,\kappa'\rangle} \ar@{->}[rr]^{0} \ar[ur]_{\kappa'} & & C 
}	
\end{aligned}
\end{equation}
Now, thanks to  \Cref{lemma:pf_and_cok} below,  the composite $q\cdot \langle0,\kappa'\rangle$    in  diagram (\ref{diag:Butter_of_ext}) is actually a pushforward of $\kappa\times_C\kappa'$ along $[1,1]$. As a consequence, the short exact sequence $(\bar\kappa,\bar\gamma)$ of diagram (\ref{diag:Baer_of_ext}) coincides with the short exact sequence $(q\cdot\langle0,\kappa'\rangle,\overline{\gamma \cdot p_1})$ of diagram  4(\ref{diag:Butter_of_ext}). 
Therefore, the identity $\Phi(E\oplus E')= \Phi(E')\cdot  \Phi(E)$ holds.

As far as the monoidal units are concerned, it is clear that the image of the unit object of $\cH^2(C,B,\xi)$   
\[
E_{\rtimes}\colon \xymatrix{B\ar[r]^-{i_B}&B\rtimes C\ar[r]^-{p_C}&C}
\]
under $\Phi$ is given by the identity butterfly of the crossed extension $I_{\xi}=(1_B,0,1_C)$:
\[
\xymatrix{
0 \ar[r] & B\ar@{==}[dd]  \ar@{=}[r] & B \ar[rr]^{0} \ar[dr]_{-i_B} & &  C\ar@{=}[r] & C \ar[r]\ar@{=}[dd] & 0 \\
& & & B\rtimes C \ar[dr]_{p_C} \ar[ur]_{p_C} \\
0 \ar[r] & B \ar@{=}[r] & B \ar[rr]_{0} \ar[ur]_{i_B} & & C\ar@{=}[r] & C \ar[r] & 0
}	
\]

\medskip

Finally, by means of the universal properties used in the composition of butterflies, one easily determines a canonical natural isomorphism 
\[
\sigma_{E,E'}\colon \Phi(E\oplus E')= \Phi(E')\cdot\Phi(E)\cong\Phi(E)\cdot\Phi(E')=\Phi(E'\oplus E)
\]
which makes $\Phi$ symmetric monoidal.
\end{proof}

\bigskip
\begin{Lemma}\label{lemma:pf_and_cok}
The square of solid arrows in diagram (\ref{diag:lemma}) below presents $q\cdot\langle0,\kappa'\rangle$ as the pushforward of $\kappa\times_C \kappa'$ along $[1,1]$. 
\begin{equation}\label{diag:lemma}
\begin{aligned}
\xymatrix@C=10ex{
B\ar[d]_{q\cdot\langle0,\kappa'\rangle }
&B\times B\ar[l]_-{[1,1]}\ar[d]^{\kappa\times_C\kappa'}
\\
E\cdot E'\ar@{-->}[d]_{\overline{\gamma \cdot p_1}}
&E\times_C E'\ar[l]_{q}\ar@{-->}[d]^r
\\
C\ar@{==}[r]&C
}
\end{aligned}
\end{equation}

\begin{proof} First we prove that the candidate pushforward square is commutative. This can be done by precomposing with the jointly epimorphic pair
\[
\xymatrix@C=10ex{
B\ar[r]^-{\langle0,1\rangle}
&B\times B
&B\ar[l]_-{\langle1,-1\rangle}\,.
}
\]
Indeed, 
\[
q\cdot\langle0,\kappa'\rangle\cdot[1,1]\cdot\langle0,1\rangle=q\cdot\langle0,\kappa'\rangle=q\cdot(\kappa\times_C\kappa')\cdot\langle0,1\rangle\,,
\]
\[
q\cdot\langle0,\kappa'\rangle\cdot[1,1]\cdot\langle1,-1\rangle=0=q\cdot\langle\kappa,-\kappa'\rangle
=q\cdot(\kappa\times_C\kappa')\cdot\langle1,-1\rangle\,.
\]

\bigskip
\noindent As a second step we prove that the dashed square  commutes. To this end, let us consider the square $(b)$ of diagram (\ref{diag:Baer_of_ext}). One has: 
\[\langle1,1\rangle\cdot r=\langle \gamma\cdot p_1,\gamma'\cdot p_2\rangle=\langle \gamma\cdot p_1,\gamma\cdot p_1\rangle=\langle1,1\rangle\cdot\gamma \cdot p_1=\langle1,1\rangle\cdot\overline{\gamma \cdot p_1}\cdot q
\]
Canceling  the monomorphism $\langle1,1\rangle$ on both sides, one obtains $r= \overline{\gamma \cdot p_1}\cdot q$.

Summarizing, diagram (\ref{diag:lemma}) shows that the pair $([1,1],q)$ is a morphism of crossed modules between normal monomorphisms with isomorphic cokernels, hence a pushforward, thanks to \cite[Theorem 2.13]{pf}.

\end{proof}
	
\end{Lemma}



\end{document}